\documentclass[11pt,reqno]{amsart}
\usepackage[T1]{fontenc}
\usepackage{lmodern}
\usepackage[a4paper,margin=28mm]{geometry}
\usepackage{amsmath,amssymb,mathtools}
\usepackage{microtype}
\usepackage{hyperref}
\hypersetup{hidelinks,
  pdftitle={Finite-sum identities and supercongruences for cubes of Catalan numbers},
  pdfauthor={}}
\allowdisplaybreaks[1]
\numberwithin{equation}{section}
\newtheorem{theorem}{Theorem}[section]
\newtheorem{lemma}[theorem]{Lemma}
\newtheorem{proposition}[theorem]{Proposition}
\newtheorem{corollary}[theorem]{Corollary}
\theoremstyle{remark}

\title[Finite identities and Catalan supercongruences]
{Finite-sum identities and supercongruences\\
for cubes of Catalan numbers}
\author{Yuqi Liu}
\address{School of Mathematical Sciences, Capital Normal University, Beijing 100048, P.R.~China}
\email{lyq204219@163.com}
\date{}

\begin{document}

\begin{abstract}
Let \(C_k=\binom{2k}{k}/(k+1)\) be the \(k\)-th Catalan number, and let
\(p\) be an odd prime. We establish exact finite-sum identities that,
together with known binomial supercongruences and reductions of Sun, yield
\[
 \sum_{k=0}^{(p-1)/2}\frac{C_k^3}{16^k}
 \equiv2-24y^2-\frac{5p^2}{2y^2}\pmod{p^3}
\]
when \(p=x^2+3y^2\equiv1\pmod3\), and
\[
 \sum_{k=0}^{(p-1)/2}\frac{C_k^3}{(-8)^k}
 \equiv-1-48y^2-\frac{33p^2}{16y^2}\pmod{p^3}
\]
when \(p=x^2+4y^2\equiv1\pmod4\). We also evaluate the corresponding
sum with denominator \(16^k(k+1)^2\), and obtain two weighted Catalan
supercongruences. The finite identities follow from the ratios of
consecutive summands, symmetry of double sums, and telescoping.
All endpoint terms are retained over the rational numbers before any
modular reduction. In each bilinear identity, one endpoint denominator
equals \(p\); its cancellation produces the necessary term of order
\(p^2\). The precise arithmetic inputs are stated separately.
\end{abstract}

\maketitle

\section{Introduction}

For a nonnegative integer \(k\), the central binomial coefficient and
the Catalan number are, respectively,
\[
 \binom{2k}{k}=\frac{(2k)!}{(k!)^2},
\qquad C_k=\frac1{k+1}\binom{2k}{k}.
\]
For the standard explicit formula for \(C_k\), see
Stanley \cite[Theorem~1.4.1, Eq.~(1.6), p.~4]{Stanley2015}.
Congruences for truncated sums of products of binomial coefficients
are often expressed in terms of representations of a prime by a binary
quadratic form. In the case relevant here, let
\[
 p=x^2+3y^2\equiv1\pmod3
\]
be a prime, with \(x,y\in\mathbb Z\).
Z.-H.~Sun \cite{SunLegendreII} obtained the unweighted evaluation
\begin{equation}\label{eq:known-A-mod-p2}
 \sum_{k=0}^{(p-1)/2}\frac{\binom{2k}{k}^3}{16^k}
 \equiv4x^2-2p=2p-12y^2\pmod{p^2}.
\end{equation}
For the sums with quadratic and cubic denominators, Sun
\cite[Theorem~4.3]{SunQuadratic2022} proved
\begin{align}
 \sum_{k=0}^{(p-1)/2}
 \frac{\binom{2k}{k}^3}{16^k(k+1)^2}
 &\equiv2p-24y^2\pmod{p^2},\label{eq:known-square-p2}\\
 \sum_{k=0}^{(p-1)/2}\frac{C_k^3}{16^k}
 &\equiv2-24y^2\pmod{p^2}.\label{eq:known-Catalan-mod-p2}
\end{align}
Here and below the theorem number in \cite{SunQuadratic2022} refers
specifically to arXiv version~4, dated 24 November 2022.
For a prime \(p=x^2+4y^2\equiv1\pmod4\), Sun
\cite[Theorem~4.6]{SunQuadratic2022} proved
\[
 \sum_{k=0}^{(p-1)/2}\frac{C_k^3}{(-8)^k}
 \equiv-1-48y^2\pmod{p^2}.
\]
For further binomial congruences with reciprocal polynomial weights,
see Sun \cite{Sun2022,Sun2023}.
Cubes of Catalan numbers with denominator \(64^k\) were studied by
Tauraso \cite{Tauraso}.

We begin with three evaluations modulo \(p^3\) associated with
the denominator \(16^k\). We derive the auxiliary evaluation by
finite-sum identities and combine it with the reduction formulas of
Sun \cite[Corollary~4.1]{SunQuadratic2022}.

\begin{theorem}\label{thm:main}
Let \(p\equiv1\pmod3\) be a prime, and let \(x,y\in\mathbb Z\) satisfy
\(p=x^2+3y^2\). Then
\begin{equation}\label{eq:main}
 \sum_{k=0}^{(p-1)/2}\frac{C_k^3}{16^k}
 \equiv2-24y^2-\frac{5p^2}{2y^2}\pmod{p^3}.
\end{equation}
\end{theorem}

The following auxiliary evaluation will be proved first.

\begin{theorem}\label{thm:aux}
Under the assumptions of Theorem~\ref{thm:main},
\begin{equation}\label{eq:aux}
 \sum_{k=0}^{(p-1)/2}
 \frac{\binom{2k}{k}^3}{16^k(k+1)}
 \equiv2p-16y^2\pmod{p^3}.
\end{equation}
\end{theorem}

Theorem~\ref{thm:aux} gives the \(p\equiv1\pmod3\) case of the first
congruence in \cite[Conjecture~5.9]{Sun2022}.
The same auxiliary evaluation also yields the quadratic-denominator
congruence in the next theorem.

\begin{theorem}\label{thm:quadratic}
Under the assumptions of Theorem~\ref{thm:main},
\begin{equation}\label{eq:quadratic-mod-p3}
 \sum_{k=0}^{(p-1)/2}
 \frac{\binom{2k}{k}^3}{16^k(k+1)^2}
 \equiv2p-24y^2-\frac{p^2}{2y^2}\pmod{p^3}.
\end{equation}
\end{theorem}

Theorem~\ref{thm:quadratic} proves the first congruence in
Sun \cite[Conjecture~5.1(ii)]{Sun2023}, since
\(\binom{2k}{k}C_k^2=\binom{2k}{k}^3/(k+1)^2\).
Theorems~\ref{thm:main} and~\ref{thm:quadratic} refine
\eqref{eq:known-Catalan-mod-p2} and~\eqref{eq:known-square-p2}, respectively.

For the negative denominator, the analogous evaluations are as follows.

\begin{theorem}\label{thm:minus-eight}
Let \(p\equiv1\pmod4\) be a prime, and let \(x,y\in\mathbb Z\)
satisfy \(p=x^2+4y^2\). Then
\begin{equation}\label{eq:minus-eight-main}
 \sum_{k=0}^{(p-1)/2}\frac{C_k^3}{(-8)^k}
 \equiv-1-48y^2-\frac{33p^2}{16y^2}\pmod{p^3}.
\end{equation}
\end{theorem}

\begin{theorem}\label{thm:minus-eight-aux}
Under the assumptions of Theorem~\ref{thm:minus-eight},
\begin{equation}\label{eq:minus-eight-aux}
 \sum_{k=0}^{(p-1)/2}
 \frac{\binom{2k}{k}^3}{(-8)^k(k+1)}
 \equiv2p-24y^2\pmod{p^3}.
\end{equation}
\end{theorem}

The unweighted evaluations modulo \(p^3\) are due to Sun and Ye
\cite[Theorems~1.2 and~1.3]{SunYe}. The weighted congruences modulo
\(p^3\) are due to Guillera and Zudilin
\cite[Eqs.~(10) and~(12)]{GuilleraZudilin}.
For the denominator \(16^k\), we use the modulo \(p^3\) consequence of
Mao and Zhang \cite[Theorem~1.2]{MaoZhang}; their stronger congruence
modulo \(p^4\) was conjectured by Z.-W.~Sun
\cite[Conjecture~5.1(ii)]{Sun2011}.
The required arithmetic inputs are stated in
Sections~\ref{subsec:arithmetic} and~\ref{subsec:minus-eight-inputs}.
The unweighted evaluations are needed modulo \(p^3\); their
reductions modulo \(p^2\) alone do not determine the final correction terms.

For clarity, all rational congruences are interpreted in the local ring
\[
 \mathbb Z_{(p)}
 =\{r/s:r,s\in\mathbb Z,\ s\ne0,\ p\nmid s\}.
\]
For \(\alpha,\beta\in\mathbb Z_{(p)}\) and an integer \(e\ge1\), the notation
\(\alpha\equiv\beta\pmod{p^e}\) means
\(\alpha-\beta\in p^e\mathbb Z_{(p)}\).
Reduction modulo \(p\) sends \(r/s\) to the residue of \(r\) times the
inverse of the residue of \(s\). An element is a unit of
\(\mathbb Z_{(p)}\) exactly when its residue modulo \(p\) is nonzero.
If \(\alpha,\beta\) are units and
\(\alpha\equiv\beta\pmod{p^e}\), then
\[
 \alpha^{-1}-\beta^{-1}
 =\frac{\beta-\alpha}{\alpha\beta}
 \in p^e\mathbb Z_{(p)}.
\]
Thus inverses preserve this congruence. We will use this fact only
after establishing that the quantities being inverted are units.

Section~\ref{sec:main-proof} contains the exact finite identities and
the complete proofs of
Theorems~\ref{thm:main}--\ref{thm:quadratic}.
Section~\ref{sec:weighted} proves the two weighted Catalan
corollaries. Section~\ref{sec:minus-eight-proof} proves
Theorems~\ref{thm:minus-eight} and~\ref{thm:minus-eight-aux},
using the notation \(b_k,F,H,J\) for that denominator.
In each proof section, \(p\) is fixed and \(m=(p-1)/2\) throughout.
Endpoint identities are retained over \(\mathbb Q\) until any
denominator divisible by \(p\) has been canceled.

\section{Proof of the main supercongruence}\label{sec:main-proof}

Throughout this section, fix a prime \(p\equiv1\pmod3\), write
\(p=x^2+3y^2\) with \(x,y\in\mathbb Z\), and put
\[
 m=\frac{p-1}{2}.
\]
Thus \(2m+1=p\).

\subsection{Exact finite identities}

For \(k=0,1,2,\ldots\), define
\begin{equation}\label{eq:ak}
 a_k=\frac{\binom{2k}{k}^3}{16^k}.
\end{equation}
The definition of the Catalan numbers gives
\[
 \sum_{k=0}^m\frac{C_k^3}{16^k}
 =\sum_{k=0}^m\frac{a_k}{(k+1)^3}.
\]

Put
\begin{equation}\label{eq:notation}
 A=\sum_{k=0}^{m}a_k,\qquad
 B=\sum_{k=0}^{m}ka_k,\qquad
 T=\sum_{k=0}^{m}\frac{a_k}{k+1}.
\end{equation}
The term \(a_{m+1}\) is also defined by \eqref{eq:ak} and will appear
as an endpoint when a summation index is shifted.

A comparison of consecutive terms in \eqref{eq:ak} yields
\begin{equation}\label{eq:recurrence}
 \frac{a_{k+1}}{a_k}
 =\frac1{16}
  \left(\frac{(2k+2)(2k+1)}{(k+1)^2}\right)^3
 =\frac{(2k+1)^3}{2(k+1)^3}.
\end{equation}
In particular, \(a_0=1\) and
\begin{equation}\label{eq:boundary}
 2(m+1)^3a_{m+1}=p^3a_m.
\end{equation}

The next two lemmas give the endpoint identities needed to evaluate
\(T\).

\begin{lemma}\label{lem:linear}
We have
\begin{equation}\label{eq:linear}
 T=2A+4B+6\sum_{k=0}^{m}k^2a_k
       -\frac{p^3a_m}{m+1}.
\end{equation}
\end{lemma}

\begin{proof}
Direct multiplication gives
\[
 (2+4k+8k^2)(k+1)-(2k+1)^3=1.
\]
Together with \eqref{eq:recurrence}, this yields
\begin{align}
 (2+4k+8k^2)a_k-2(k+1)^2a_{k+1}
 &=\left(2+4k+8k^2-\frac{(2k+1)^3}{k+1}\right)a_k \notag\\
 &=\frac{a_k}{k+1}.
 \label{eq:linear-term}
\end{align}
Summing \eqref{eq:linear-term} over \(0\le k\le m\), we find
\[
 T=2A+4B+8\sum_{k=0}^{m}k^2a_k
       -2\sum_{k=0}^{m}(k+1)^2a_{k+1}.
\]
Changing the index in the last sum gives
\[
 \sum_{k=0}^{m}(k+1)^2a_{k+1}
 =\sum_{k=1}^{m+1}k^2a_k
 =\sum_{k=0}^{m}k^2a_k+(m+1)^2a_{m+1}.
\]
It follows that
\[
 T=2A+4B+6\sum_{k=0}^{m}k^2a_k-2(m+1)^2a_{m+1}.
\]
Substituting \eqref{eq:boundary} proves \eqref{eq:linear}.
\end{proof}

\begin{lemma}\label{lem:bilinear}
We have
\begin{equation}\label{eq:bilinear}
 AT-A^2-3B^2
 =p^3a_m
   \sum_{k=0}^{m}a_k
   \left(\frac1{m+k+1}-\frac1{m+1}\right).
\end{equation}
\end{lemma}

\begin{proof}
We expand products of finite sums as double sums. By
\eqref{eq:recurrence},
\begin{align}
 \sum_{i,j=0}^{m}\frac{(2i+1)^3a_i a_j}{i+j+1}
 &=2\sum_{i=0}^{m}\sum_{j=0}^{m}
       \frac{(i+1)^3a_{i+1}a_j}{i+j+1} \notag\\
 &=2\sum_{i=1}^{m+1}\sum_{j=0}^{m}
       \frac{i^3a_i a_j}{i+j} \notag\\
 &=2\sum_{i=1}^{m}\sum_{j=0}^{m}
       \frac{i^3a_i a_j}{i+j}
   +p^3a_m\sum_{j=0}^{m}\frac{a_j}{m+j+1}.
 \label{eq:double-shift}
\end{align}
The last equality follows by isolating the term with \(i=m+1\) and
then using \eqref{eq:boundary}.
The change of index moves the interval \(0\le i\le m\) to
\(1\le i\le m+1\). Thus the final row must be kept separately when the
remaining double sum is returned to the original square.

We next compute both double sums in \eqref{eq:double-shift}.
Since the summation region on its left is unchanged when \(i\) and \(j\)
are interchanged, taking the average of the two expressions gives
\begin{align}
 \sum_{i,j=0}^{m}\frac{(2i+1)^3a_i a_j}{i+j+1}
 &=\sum_{i,j=0}^{m}
 \frac{(2i+1)^3+(2j+1)^3}{2(i+j+1)}a_i a_j \notag\\
 &=\sum_{i,j=0}^{m}
 \bigl(1+2(i+j)+4(i^2-ij+j^2)\bigr)a_i a_j.
 \label{eq:symmetric-left}
\end{align}
Here we used the algebraic identity
\[
 (2i+1)^3+(2j+1)^3
 =2(i+j+1)\bigl(1+2(i+j)+4(i^2-ij+j^2)\bigr).
\]
The definitions \eqref{eq:notation} imply
\begin{align*}
 \sum_{i,j=0}^{m}a_i a_j&=A^2,\\
 \sum_{i,j=0}^{m}(i+j)a_i a_j&=2AB,\\
 \sum_{i,j=0}^{m}ij\,a_i a_j&=B^2,\\
 \sum_{i,j=0}^{m}(i^2+j^2)a_i a_j
     &=2A\sum_{k=0}^{m}k^2a_k.
\end{align*}
For example, the contribution of \(i\) factors as
\((\sum_{i=0}^{m}ia_i)(\sum_{j=0}^{m}a_j)=BA\);
the contribution of \(j\) is \(AB\). The other evaluations follow by
the same separation of the two indices.
Substituting these expressions into \eqref{eq:symmetric-left}, we obtain
\begin{equation}\label{eq:left-evaluated}
 \sum_{i,j=0}^{m}\frac{(2i+1)^3a_i a_j}{i+j+1}
 =A^2+4AB+8A\sum_{k=0}^{m}k^2a_k-4B^2.
\end{equation}

For the internal sum on the right of \eqref{eq:double-shift}, the same
symmetry gives
\begin{align}
 2\sum_{i=1}^{m}\sum_{j=0}^{m}\frac{i^3a_i a_j}{i+j}
 &=\sum_{\substack{0\le i,j\le m\\i+j>0}}
       \frac{i^3+j^3}{i+j}a_i a_j \notag\\
 &=\sum_{i,j=0}^{m}(i^2-ij+j^2)a_i a_j \notag\\
 &=2A\sum_{k=0}^{m}k^2a_k-B^2.
 \label{eq:right-evaluated}
\end{align}
To justify the first equality, one can first add the terms with \(i=0\)
and \(j>0\), whose numerators vanish. The region is then the square
with \((0,0)\) removed, so interchanging the indices and adding the two
expressions replaces \(2i^3\) by \(i^3+j^3\).
Indeed, \(i^3+j^3=(i+j)(i^2-ij+j^2)\). In the second line the
\((i,j)=(0,0)\) term can be restored because the polynomial
\(i^2-ij+j^2\) vanishes there.

Combining \eqref{eq:double-shift}, \eqref{eq:left-evaluated}, and
\eqref{eq:right-evaluated}, and collecting terms, yields
\begin{equation}\label{eq:mixed-moment}
 A^2+4AB+6A\sum_{k=0}^{m}k^2a_k-3B^2
 =p^3a_m\sum_{k=0}^{m}\frac{a_k}{m+k+1}.
\end{equation}
On the other hand, multiplying \eqref{eq:linear} by \(A\) and subtracting
\(A^2+3B^2\) gives
\begin{align*}
 AT-A^2-3B^2
 &=A^2+4AB+6A\sum_{k=0}^{m}k^2a_k-3B^2
     -\frac{p^3a_mA}{m+1}\\
 &=p^3a_m
   \left(\sum_{k=0}^{m}\frac{a_k}{m+k+1}-\frac{A}{m+1}\right).
\end{align*}
Expanding \(A=\sum_{k=0}^{m}a_k\) in the last term proves
\eqref{eq:bilinear}.
\end{proof}

\medskip
\noindent\textbf{Known reductions of Sun.}
Taking Sun's denominator parameter to be \(1\) in
\cite[Corollary~4.1]{SunQuadratic2022}, we obtain
\begin{equation}\label{eq:quadratic}
 \sum_{k=0}^{m}\frac{a_k}{(k+1)^2}
 \equiv-4A+6B+6T\pmod{p^3}
\end{equation}
and
\begin{equation}\label{eq:cubic}
 \sum_{k=0}^{m}\frac{a_k}{(k+1)^3}
 \equiv2-18A+36B+24T\pmod{p^3}.
\end{equation}

\subsection{Arithmetic evaluations and proofs}\label{subsec:arithmetic}

Since \(1\le k+1<p\) for \(0\le k\le m\), the denominators
in \(A,B,T\) and in the Catalan sum are coprime to \(p\).

\begin{proposition}\label{prop:p-adic-bilinear}
The sums \(A,B,T\) satisfy
\begin{equation}\label{eq:p-adic-bilinear}
 AT-A^2-3B^2\equiv p^2\pmod{p^3}.
\end{equation}
\end{proposition}

\begin{proof}
Equation~\eqref{eq:boundary} gives
\[
 a_{m+1}=\frac{p^3a_m}{2(m+1)^3}\equiv0\pmod{p^3}.
\]
In \eqref{eq:bilinear}, the factor outside the sum is \(p^3a_m\).
Separating the summand with \(k=m\) gives the exact equality
\[
 AT-A^2-3B^2
 =p^3a_m\sum_{k=0}^{m-1}\frac{a_k}{m+k+1}
   +p^2a_m^2-\frac{p^3a_mA}{m+1}.
\]
For \(0\le k<m\), we have \(m+1\le m+k+1\le p-1\).
Thus the first and third terms on the right are multiples of \(p^3\)
in \(\mathbb Z_{(p)}\). The middle term comes from cancelling the
denominator \(2m+1=p\); it must be retained. Consequently,
\begin{align}
 AT-A^2-3B^2
 &\equiv p^2a_m^2\pmod{p^3}.
 \label{eq:resonant-boundary}
\end{align}
Finally,
\[
 \binom{p-1}{m}
 =\prod_{r=1}^{m}\frac{p-r}{r}\equiv(-1)^m\pmod p,
 \qquad 16^{p-1}\equiv1\pmod p.
\]
It follows that
\[
 a_m^2=\frac{\binom{p-1}{m}^{6}}{16^{p-1}}\equiv1\pmod p.
\]
Substitution into \eqref{eq:resonant-boundary} proves the proposition.
\end{proof}

We recall explicitly the two arithmetic results used below.
Our fixed prime \(p=x^2+3y^2\equiv1\pmod3\) satisfies \(p\ge7\) and
\(p\nmid xy\). To see this, if \(p\mid x\), the equality
\(x^2+3y^2=p\) implies \(p\mid3y^2\), and hence \(p\mid y\), since
\(p\ne3\). The left side would then be divisible by \(p^2\), a
contradiction. If \(p\mid y\), the same equality first gives
\(p\mid x\), leading to the same contradiction. Thus all fractions
with denominator a product of \(2,3,x,y\) lie in \(\mathbb Z_{(p)}\).

\medskip
\noindent\textbf{Known result of Sun and Ye.}
For a prime \(p\equiv1\pmod3\), written as \(p=x^2+3y^2\),
\cite[Theorem~1.2, first congruence in Eq.~(1.8)]{SunYe} states
\begin{equation}\label{eq:known-A}
 A\equiv4x^2-2p-\frac{p^2}{4x^2}\pmod{p^3}.
\end{equation}
\medskip
\noindent\textbf{Known result of Mao and Zhang.}
Reducing the congruence in
\cite[Theorem~1.2, Eq.~(1.3)]{MaoZhang} modulo \(p^3\) gives
\begin{equation}\label{eq:known-weighted}
 A+3B=\sum_{k=0}^{m}(3k+1)
       \frac{\binom{2k}{k}^3}{16^k}
 \equiv p\pmod{p^3}.
\end{equation}

\begin{proof}[Proof of Theorem~\ref{thm:aux}]
We have \(p\nmid x\), so \eqref{eq:known-A} implies
\(A\equiv4x^2\not\equiv0\pmod p\). Hence \(A\) is invertible modulo
\(p^3\). By Proposition~\ref{prop:p-adic-bilinear},
\[
 T\equiv A+\frac{3B^2+p^2}{A}\pmod{p^3}.
\]
From \eqref{eq:known-weighted}, \(B\equiv(p-A)/3\pmod{p^3}\).
Consequently,
\begin{align}
 T
 &\equiv A+\frac{(p-A)^2}{3A}+\frac{p^2}{A} \notag\\
 &=A+\frac{A^2-2pA+4p^2}{3A}\notag\\
 &=A+\frac{A}{3}-\frac{2p}{3}+\frac{4p^2}{3A}\notag\\
 &=\frac43A-\frac23p+\frac{4p^2}{3A}\pmod{p^3}.
 \label{eq:T-from-A}
\end{align}
The last term has a factor \(p^2\); using \(A\equiv4x^2\pmod p\) there,
we obtain
\[
 \frac{4p^2}{3A}\equiv\frac{p^2}{3x^2}\pmod{p^3}.
\]
This replacement needs only precision modulo \(p\), because
\(A^{-1}-(4x^2)^{-1}=(4x^2-A)/(4x^2A)\) is a multiple of \(p\)
in \(\mathbb Z_{(p)}\). Multiplication by \(p^2\) makes the resulting
error a multiple of \(p^3\).
Substitute \eqref{eq:known-A} into \eqref{eq:T-from-A}:
\begin{align*}
 T
 &\equiv\frac43\left(4x^2-2p-\frac{p^2}{4x^2}\right)
       -\frac23p+\frac{p^2}{3x^2}\\
 &=\frac{16}{3}x^2-\frac{10}{3}p\\
 &=2p-16y^2\pmod{p^3}.
\end{align*}
The two terms involving \(p^2/x^2\) cancel, and the last equality uses
\(x^2=p-3y^2\). This proves \eqref{eq:aux}.
\end{proof}

\begin{proof}[Proof of Theorem~\ref{thm:quadratic}]
Equation~\eqref{eq:known-weighted} gives
\(6B\equiv2p-2A\pmod{p^3}\). Substituting this into
\eqref{eq:quadratic}, we obtain
\begin{align}
 \sum_{k=0}^{m}\frac{a_k}{(k+1)^2}
 &\equiv-4A+6B+6T\notag\\
 &\equiv-6A+2p+6T\pmod{p^3}.
 \label{eq:quadratic-reduction}
\end{align}
Substitute \eqref{eq:known-A} and the evaluation of \(T\) in
Theorem~\ref{thm:aux}:
\begin{align*}
 \sum_{k=0}^{m}\frac{a_k}{(k+1)^2}
 &\equiv-6\left(4x^2-2p-\frac{p^2}{4x^2}\right)+2p\\
 &\qquad+6(2p-16y^2)\\
 &=-24x^2+26p-96y^2+\frac{3p^2}{2x^2}\\
 &=2p-24y^2+\frac{3p^2}{2x^2}\\
 &\equiv2p-24y^2-\frac{p^2}{2y^2}\pmod{p^3}.
\end{align*}
The penultimate equality uses \(x^2=p-3y^2\).
For the final congruence, the exact identity
\[
 \frac{3p^2}{2x^2}+\frac{p^2}{2y^2}
 =\frac{p^2(3y^2+x^2)}{2x^2y^2}
 =\frac{p^3}{2x^2y^2}
\]
shows that \(3p^2/(2x^2)\equiv-p^2/(2y^2)\pmod{p^3}\), since
\(p\nmid2x^2y^2\).
This proves \eqref{eq:quadratic-mod-p3}.
\end{proof}

\begin{proof}[Proof of Theorem~\ref{thm:main}]
Since \(C_k^3/16^k=a_k/(k+1)^3\), \eqref{eq:cubic} yields
\begin{align}
 \sum_{k=0}^{m}\frac{C_k^3}{16^k}
 &\equiv2-18A+36B+24T \notag\\
 &\equiv2-30A+12p+24T\pmod{p^3},
 \label{eq:main-reduction}
\end{align}
where the second line uses
\(36B\equiv12(p-A)=12p-12A\pmod{p^3}\), by
\eqref{eq:known-weighted}.

Apply \eqref{eq:known-A} and Theorem~\ref{thm:aux} to
\eqref{eq:main-reduction}:
\begin{align*}
 \sum_{k=0}^{m}\frac{C_k^3}{16^k}
 &\equiv2-30\left(4x^2-2p-\frac{p^2}{4x^2}\right)
          +12p+24(2p-16y^2)\\
 &=2-120x^2-384y^2+120p+\frac{15p^2}{2x^2}\\
 &=2-24y^2+\frac{15p^2}{2x^2}\pmod{p^3}.
\end{align*}
Here \(-120x^2+120p=360y^2\).
For the remaining fraction, the exact identity
\[
 \frac{15p^2}{2x^2}+\frac{5p^2}{2y^2}
 =\frac{5p^2(x^2+3y^2)}{2x^2y^2}
 =\frac{5p^3}{2x^2y^2}
\]
shows that \(\frac{15p^2}{2x^2}\equiv-\frac{5p^2}{2y^2}\pmod{p^3}\),
since \(p\nmid2x^2y^2\). This proves \eqref{eq:main}.
\end{proof}

\section{Two weighted Catalan supercongruences}\label{sec:weighted}

The evaluations with quadratic and cubic denominators established in
Section~\ref{sec:main-proof} give two direct consequences. In this
section, the sequence \(a_k\) is still the one in \eqref{eq:ak} and
\(m=(p-1)/2\).

\begin{corollary}\label{cor:weight-one}
Let \(p\equiv1\pmod3\) be a prime, and write \(p=x^2+3y^2\). Then
\begin{equation}\label{eq:weight-one}
 \sum_{k=0}^{(p-1)/2}\frac{kC_k^3}{16^k}
 \equiv2p-2+\frac{2p^2}{y^2}\pmod{p^3}.
\end{equation}
\end{corollary}

\begin{proof}
Put \(m=(p-1)/2\). The elementary identity
\[
 \frac{k}{(k+1)^3}
 =\frac1{(k+1)^2}-\frac1{(k+1)^3}
\]
and \(C_k^3/16^k=a_k/(k+1)^3\) imply
\[
 \sum_{k=0}^{m}\frac{kC_k^3}{16^k}
 =\sum_{k=0}^{m}\frac{a_k}{(k+1)^2}
  -\sum_{k=0}^{m}\frac{C_k^3}{16^k}.
\]
Apply Theorem~\ref{thm:quadratic} and
Theorem~\ref{thm:main} to the two sums on the right:
\begin{align*}
 \sum_{k=0}^{m}\frac{kC_k^3}{16^k}
 &\equiv\left(2p-24y^2-\frac{p^2}{2y^2}\right)\\
 &\qquad-\left(2-24y^2-\frac{5p^2}{2y^2}\right)\\
 &=2p-2+\frac{2p^2}{y^2}\pmod{p^3}.
\end{align*}
\end{proof}

\begin{corollary}\label{cor:weight-two}
Let \(p\equiv1\pmod3\) be a prime, and write \(p=x^2+3y^2\). Then
\begin{equation}\label{eq:weight-two}
 \sum_{k=0}^{(p-1)/2}\frac{k^2C_k^3}{16^k}
 \equiv2-2p+8y^2-\frac{3p^2}{2y^2}\pmod{p^3}.
\end{equation}
\end{corollary}

\begin{proof}
Expanding \(k^2=((k+1)-1)^2\) and dividing by \((k+1)^3\) gives
\[
 \frac{k^2}{(k+1)^3}
 =\frac1{k+1}-\frac2{(k+1)^2}+\frac1{(k+1)^3}.
\]
Multiplying by \(a_k\), summing over \(0\le k\le m\), and using
\(C_k^3/16^k=a_k/(k+1)^3\), we obtain the exact equality
\[
 \sum_{k=0}^m\frac{k^2C_k^3}{16^k}
 =T-2\sum_{k=0}^m\frac{a_k}{(k+1)^2}
     +\sum_{k=0}^m\frac{C_k^3}{16^k}.
\]
Theorems~\ref{thm:aux}, \ref{thm:quadratic}, and \ref{thm:main}
therefore give
\begin{align*}
 \sum_{k=0}^m\frac{k^2C_k^3}{16^k}
 &\equiv (2p-16y^2)
       -2\left(2p-24y^2-\frac{p^2}{2y^2}\right)\\
 &\qquad+\left(2-24y^2-\frac{5p^2}{2y^2}\right)\\
 &=2+(2p-4p)+(-16+48-24)y^2
            +\left(1-\frac52\right)\frac{p^2}{y^2}\\
 &=2-2p+8y^2-\frac{3p^2}{2y^2}\pmod{p^3}.
\end{align*}
\end{proof}

\section{The negative denominator}\label{sec:minus-eight-proof}

Throughout this section, fix a prime \(p\equiv1\pmod4\), write
\(p=x^2+4y^2\) with \(x,y\in\mathbb Z\), and put
\[
 m=\frac{p-1}{2}.
\]
Thus \(2m+1=p\). Define
\begin{equation}\label{eq:bk}
 b_k=\frac{\binom{2k}{k}^3}{(-8)^k}
 \qquad(k=0,1,2,\ldots),
\end{equation}
and put
\begin{equation}\label{eq:negative-notation}
 F=\sum_{k=0}^m b_k,\qquad
 H=\sum_{k=0}^m kb_k,\qquad
 J=\sum_{k=0}^m\frac{b_k}{k+1}.
\end{equation}
The next term \(b_{m+1}\) is defined by \eqref{eq:bk} and occurs
as a boundary term in the index shifts below.
Directly from the definition of \(C_k\),
\begin{equation}\label{eq:negative-catalan-sum}
 \sum_{k=0}^m\frac{C_k^3}{(-8)^k}
 =\sum_{k=0}^m\frac{b_k}{(k+1)^3}.
\end{equation}
\subsection{The recurrence and the first finite identity}

The quotient of two consecutive terms in \eqref{eq:bk} is
\begin{align}
 \frac{b_{k+1}}{b_k}
 &=\frac1{-8}
   \left(\frac{(2k+2)(2k+1)}{(k+1)^2}\right)^3\notag\\
 &=-\frac1{8}\left(\frac{2(2k+1)}{k+1}\right)^3
 =-\frac{(2k+1)^3}{(k+1)^3}.
 \label{eq:negative-recurrence}
\end{align}
Equivalently,
\begin{equation}\label{eq:negative-recurrence-product}
 (2k+1)^3b_k=-(k+1)^3b_{k+1}.
\end{equation}
At the upper limit this gives
\begin{equation}\label{eq:negative-boundary}
 (m+1)^3b_{m+1}=-p^3b_m.
\end{equation}

\begin{lemma}\label{lem:negative-linear}
We have
\begin{equation}\label{eq:negative-linear}
 J=2F+4H+9\sum_{k=0}^m k^2b_k
       -\frac{p^3b_m}{m+1}.
\end{equation}
\end{lemma}

\begin{proof}
The polynomial identity used in Lemma~\ref{lem:linear} is
\[
 (2+4k+8k^2)(k+1)-(2k+1)^3=1.
\]
Multiply by \(b_k/(k+1)\), and then apply
\eqref{eq:negative-recurrence-product}. This gives
\begin{align}
 \frac{b_k}{k+1}
 &=(2+4k+8k^2)b_k-\frac{(2k+1)^3b_k}{k+1}\notag\\
 &=(2+4k+8k^2)b_k+(k+1)^2b_{k+1}.
 \label{eq:negative-linear-term}
\end{align}
Summing over \(0\le k\le m\),
\[
 J=2F+4H+8\sum_{k=0}^m k^2b_k
                 +\sum_{k=0}^m(k+1)^2b_{k+1}.
\]
For the last sum, the change of variable \(r=k+1\) gives
\[
 \sum_{k=0}^m(k+1)^2b_{k+1}
 =\sum_{r=1}^{m+1}r^2b_r
 =\sum_{k=0}^m k^2b_k+(m+1)^2b_{m+1}.
\]
The \(k=0\) term in the middle quadratic sum is zero.
Consequently,
\[
 J=2F+4H+9\sum_{k=0}^m k^2b_k+(m+1)^2b_{m+1}.
\]
Finally, \eqref{eq:negative-boundary} implies
\((m+1)^2b_{m+1}=-p^3b_m/(m+1)\), proving
\eqref{eq:negative-linear}.
\end{proof}

\subsection{A bilinear identity with an explicit endpoint}

\begin{lemma}\label{lem:negative-bilinear}
We have
\begin{equation}\label{eq:negative-bilinear}
 \begin{split}
 FJ-F^2-\frac92H^2
  &=p^3b_m\left(
       \sum_{j=0}^m\frac{b_j}{m+j+1}-\frac{F}{m+1}\right).
 \end{split}
\end{equation}
\end{lemma}

\begin{proof}
Consider the same double sum in two ways. First,
\eqref{eq:negative-recurrence-product} gives
\begin{align}
 &\sum_{i=0}^m\sum_{j=0}^m
        \frac{(2i+1)^3b_i b_j}{i+j+1}\notag\\
 &\qquad=-\sum_{i=0}^m\sum_{j=0}^m
        \frac{(i+1)^3b_{i+1}b_j}{i+j+1}\notag\\
 &\qquad=-\sum_{r=1}^{m+1}\sum_{j=0}^m
        \frac{r^3b_r b_j}{r+j}\notag\\
 &\qquad=-\sum_{i=1}^m\sum_{j=0}^m
        \frac{i^3b_i b_j}{i+j}
       -(m+1)^3b_{m+1}\sum_{j=0}^m\frac{b_j}{m+j+1}\notag\\
 &\qquad=-\sum_{i=1}^m\sum_{j=0}^m
        \frac{i^3b_i b_j}{i+j}
       +p^3b_m\sum_{j=0}^m\frac{b_j}{m+j+1}.
 \label{eq:negative-double-shift}
\end{align}
In the second equality we set \(r=i+1\), so that
\(i+j+1=r+j\), and the interval \(0\le i\le m\) becomes
\(1\le r\le m+1\). We then separate the row \(r=m+1\)
and rename the remaining dummy index \(r\) as \(i\).
The last equality uses \eqref{eq:negative-boundary}; its minus sign
changes the endpoint contribution to a positive term.

We next evaluate the left side of \eqref{eq:negative-double-shift}.
Its summation region is unchanged by interchanging \(i\) and \(j\),
so averaging the two expressions yields
\begin{align}
 \sum_{i,j=0}^m\frac{(2i+1)^3b_i b_j}{i+j+1}
 &=\sum_{i,j=0}^m
   \frac{(2i+1)^3+(2j+1)^3}{2(i+j+1)}b_i b_j\notag\\
 &=\sum_{i,j=0}^m
   \bigl(1+2(i+j)+4(i^2-ij+j^2)\bigr)b_i b_j.
 \label{eq:negative-symmetric-left}
\end{align}
The cubic factorization is the one used in
\eqref{eq:symmetric-left}. Separating the two indices gives
\begin{align*}
 \sum_{i,j=0}^m b_i b_j&=F^2,\\
 \sum_{i,j=0}^m(i+j)b_i b_j&=2FH,\\
 \sum_{i,j=0}^m ij\,b_i b_j&=H^2,\\
 \sum_{i,j=0}^m(i^2+j^2)b_i b_j
     &=2F\sum_{k=0}^m k^2b_k.
\end{align*}
Thus
\begin{equation}\label{eq:negative-left-evaluated}
 \sum_{i,j=0}^m\frac{(2i+1)^3b_i b_j}{i+j+1}
 =F^2+4FH+8F\sum_{k=0}^m k^2b_k-4H^2.
\end{equation}

To evaluate the internal sum on the right of
\eqref{eq:negative-double-shift}, separate the terms with \(j=0\).
Since \(b_0=1\),
\begin{align}
 \sum_{i=1}^m\sum_{j=0}^m\frac{i^3b_i b_j}{i+j}
 &=\sum_{i=1}^m i^2b_i
   +\sum_{i=1}^m\sum_{j=1}^m\frac{i^3b_i b_j}{i+j}.
 \label{eq:negative-zero-row}
\end{align}
The two indices in the last sum both run from \(1\) to \(m\).
We may therefore interchange them and take the average:
\begin{align*}
 \sum_{i,j=1}^m\frac{i^3b_i b_j}{i+j}
 &=\frac12\sum_{i,j=1}^m\frac{i^3+j^3}{i+j}b_i b_j\\
 &=\frac12\sum_{i,j=1}^m(i^2-ij+j^2)b_i b_j\\
 &=\left(\sum_{i=1}^m i^2b_i\right)
      \left(\sum_{j=1}^m b_j\right)
      -\frac12\left(\sum_{i=1}^m ib_i\right)^2\\
 &=(F-1)\sum_{k=0}^m k^2b_k-\frac12H^2.
\end{align*}
Here we used \(i^3+j^3=(i+j)(i^2-ij+j^2)\),
\(\sum_{j=1}^m b_j=F-b_0=F-1\), and the fact that the terms
with \(k=0\) in the first and second moments vanish.
Adding back the first term in \eqref{eq:negative-zero-row} gives
\begin{equation}\label{eq:negative-right-evaluated}
 \sum_{i=1}^m\sum_{j=0}^m\frac{i^3b_i b_j}{i+j}
 =F\sum_{k=0}^m k^2b_k-\frac12H^2.
\end{equation}
Every explicit denominator \(i+j\) above is positive. In particular,
no fraction with denominator \(i+j=0\) has been introduced.

Substitution of \eqref{eq:negative-left-evaluated} and
\eqref{eq:negative-right-evaluated} into
\eqref{eq:negative-double-shift} gives
\begin{align*}
 F^2+4FH+8F\sum_{k=0}^m k^2b_k-4H^2
 &=-F\sum_{k=0}^m k^2b_k+\frac12H^2\\
 &\quad+p^3b_m\sum_{j=0}^m\frac{b_j}{m+j+1}.
\end{align*}
Moving the quadratic moment and the term \(H^2/2\) to the left,
we obtain
\begin{equation}\label{eq:negative-mixed-moment}
 F^2+4FH+9F\sum_{k=0}^m k^2b_k-\frac92H^2
 =p^3b_m\sum_{j=0}^m\frac{b_j}{m+j+1}.
\end{equation}
On the other hand, Lemma~\ref{lem:negative-linear} implies
\begin{align*}
 FJ-F^2-\frac92H^2
 &=F^2+4FH+9F\sum_{k=0}^m k^2b_k-\frac92H^2\\
 &\quad-\frac{p^3b_mF}{m+1}.
\end{align*}
Replacing the first four terms on the right by
\eqref{eq:negative-mixed-moment} proves
\eqref{eq:negative-bilinear}.
\end{proof}

\begin{proposition}\label{prop:negative-bilinear-mod}
The sums \(F,H,J\) satisfy
\begin{equation}\label{eq:negative-bilinear-mod}
 FJ-F^2-\frac92H^2\equiv p^2\pmod{p^3}.
\end{equation}
\end{proposition}

\begin{proof}
Equation~\eqref{eq:negative-boundary} gives
\begin{equation}\label{eq:negative-endpoint-mod}
 b_{m+1}=-\frac{p^3b_m}{(m+1)^3}\equiv0\pmod{p^3}.
\end{equation}
Here \(m+1=(p+1)/2\) lies strictly between \(0\) and \(p\),
and \(b_m\in\mathbb Z_{(p)}\), so the displayed congruence is valid.

Apply Lemma~\ref{lem:negative-bilinear} and separate the term
\(j=m\). Its denominator is \(m+m+1=p\), so cancellation gives
\begin{align}
 FJ-F^2-\frac92H^2
 &=p^2b_m^2\notag\\
 &\quad+p^3b_m\left(
      \sum_{j=0}^{m-1}\frac{b_j}{m+j+1}-\frac{F}{m+1}\right).
 \label{eq:negative-endpoint-separated}
\end{align}
For \(0\le j\le m-1\),
\[
 m+1\le m+j+1\le2m=p-1.
\]
Every denominator in parentheses is therefore coprime to \(p\),
including the denominator \(m+1\) in the last term.
Since each \(b_j\) and \(F\) belongs to \(\mathbb Z_{(p)}\), the
whole parenthesis belongs to that ring. Its product with \(p^3b_m\)
vanishes modulo \(p^3\). Consequently,
\[
 FJ-F^2-\frac92H^2\equiv p^2b_m^2\pmod{p^3}.
\]
To determine \(b_m^2\) modulo \(p\), use
\[
 \binom{p-1}{m}
 =\prod_{r=1}^m\frac{p-r}{r}\equiv(-1)^m\pmod p
\]
and Fermat's little theorem \((-8)^{p-1}\equiv1\pmod p\).
As \(2m=p-1\), these give
\[
 b_m^2
 =\frac{\binom{p-1}{m}^6}{(-8)^{2m}}
 =\frac{\binom{p-1}{m}^6}{(-8)^{p-1}}
 \equiv\frac{(-1)^{6m}}1=1\pmod p.
\]
Thus \(p^2b_m^2\equiv p^2\pmod{p^3}\), which proves
\eqref{eq:negative-bilinear-mod}.
\end{proof}

\subsection{The cubic-denominator reduction}

Taking Sun's denominator parameter to be \(-\tfrac12\) in
\cite[Corollary~4.1]{SunQuadratic2022} gives
\begin{equation}\label{eq:negative-cubic}
 \sum_{k=0}^m\frac{b_k}{(k+1)^3}
 \equiv-1-15F+54H+24J\pmod{p^3}.
\end{equation}

\subsection{The arithmetic inputs}\label{subsec:minus-eight-inputs}

The fixed prime \(p=x^2+4y^2\equiv1\pmod4\) satisfies \(p\ge5\) and
\(p\nmid xy\). Indeed, if \(p\mid x\), reduction of
\(x^2+4y^2=p\) modulo \(p\) gives \(p\mid y\), because
\(p\ne2\); the defining equality would then have its left side
divisible by \(p^2\). If \(p\mid y\), the same reduction first
gives \(p\mid x\), leading to the same contradiction.
Consequently, the fractions below with denominators involving
\(2,3,x,y\) lie in \(\mathbb Z_{(p)}\).

\begin{proposition}\label{prop:negative-arithmetic}
Under these assumptions,
\begin{align}
 F&\equiv4x^2-2p-\frac{p^2}{4x^2}\pmod{p^3},
 \label{eq:known-F}\\
 F+3H&\equiv p\pmod{p^3}.
 \label{eq:known-negative-weighted}
\end{align}
\end{proposition}

\begin{proof}
Equation~\eqref{eq:known-F} is exactly
Sun and Ye \cite[Theorem~1.3, Eq.~(1.10)]{SunYe}.
For the second congruence we give the normalization of the cited
formula in detail. For a number \(a\) and an integer \(k\ge0\),
the rising factorial is
\[
 (a)_0=1,\qquad
 (a)_k=a(a+1)\cdots(a+k-1)\quad(k\ge1).
\]
Thus \((1)_k=k!\). For any prime \(p>2\), Guillera and Zudilin
\cite[Eq.~(12)]{GuilleraZudilin} proved
\begin{equation}\label{eq:GZ-negative}
 \sum_{k=0}^{p-1}
 \frac{(\tfrac12)_k^3}{(1)_k^3}(3k+1)(-1)^k2^{3k}
 \equiv(-1)^{(p-1)/2}p\pmod{p^3}.
\end{equation}
The odd-even factorization of \((2k)!\) gives
\[
 (\tfrac12)_k=\frac{1\cdot3\cdots(2k-1)}{2^k}
 =\frac{(2k)!}{4^k k!},
 \qquad
 \frac{(\tfrac12)_k}{(1)_k}=\frac1{4^k}\binom{2k}{k}.
\]
These formulas include \(k=0\), with the empty product equal to \(1\).
Consequently,
\begin{align*}
 \frac{(\tfrac12)_k^3}{(1)_k^3}(-1)^k2^{3k}
 &=\binom{2k}{k}^3\frac{(-1)^k8^k}{64^k}\\
 &=\binom{2k}{k}^3\frac{(-1)^k}{8^k}
 =\frac{\binom{2k}{k}^3}{(-8)^k}=b_k.
\end{align*}
Equation~\eqref{eq:GZ-negative} therefore becomes
\[
 \sum_{k=0}^{p-1}(3k+1)b_k
 \equiv(-1)^{(p-1)/2}p\pmod{p^3}.
\]

We justify the shorter upper limit used in \(F\) and \(H\).
If \(m<k\le p-1\), then \(k<p\) and
\(p<2k<2p\). The factorial \((2k)!\) contains exactly one
factor divisible by \(p\), while \((k!)^2\) contains none. Hence
\(p\mid\binom{2k}{k}\), so
\[
 b_k\equiv0\pmod{p^3}\qquad(m<k\le p-1),
\]
because \((-8)^k\) is coprime to \(p\).
The same holds after multiplication by \(3k+1\), and all these
tail terms can be removed from the sum modulo \(p^3\).
Finally, \(p\equiv1\pmod4\) makes \((p-1)/2\) even, so
\((-1)^{(p-1)/2}=1\). We obtain
\[
 \sum_{k=0}^{m}(3k+1)b_k
 =F+3H\equiv p\pmod{p^3},
\]
which proves \eqref{eq:known-negative-weighted}.
\end{proof}

\subsection{Evaluation of the auxiliary sum and the Catalan sum}

\begin{proof}[Proof of Theorem~\ref{thm:minus-eight-aux}]
By \eqref{eq:known-F},
\(F\equiv4x^2\not\equiv0\pmod p\), so \(F\) is a unit in
\(\mathbb Z_{(p)}\). We may therefore divide
\eqref{eq:negative-bilinear-mod} by \(F\), obtaining
\begin{equation}\label{eq:J-first-reduction}
 J\equiv F+\frac{p^2+\tfrac92H^2}{F}\pmod{p^3}.
\end{equation}
Equation~\eqref{eq:known-negative-weighted} gives
\(H\equiv(p-F)/3\pmod{p^3}\), since \(p\ne3\).
Substitution into \eqref{eq:J-first-reduction} yields
\begin{align}
 J
 &\equiv F+\frac{p^2}{F}
       +\frac{9}{2F}\left(\frac{p-F}{3}\right)^2\notag\\
 &=F+\frac{p^2}{F}+\frac{(p-F)^2}{2F}\notag\\
 &=F+\frac{p^2}{F}+\frac{p^2-2pF+F^2}{2F}\notag\\
 &=\frac32F-p+\frac{3p^2}{2F}\pmod{p^3}.
 \label{eq:J-from-F}
\end{align}
Since \(F\equiv4x^2\pmod p\) and both elements are units,
\[
 \frac1F-\frac1{4x^2}
 =\frac{4x^2-F}{4x^2F}\in p\mathbb Z_{(p)}.
\]
Multiplying by \(3p^2/2\) proves the replacement
\[
 \frac{3p^2}{2F}\equiv\frac{3p^2}{8x^2}\pmod{p^3}.
\]
Now substitute \eqref{eq:known-F} into \eqref{eq:J-from-F}:
\begin{align*}
 J
 &\equiv\frac32\left(4x^2-2p-\frac{p^2}{4x^2}\right)
         -p+\frac{3p^2}{8x^2}\\
 &=6x^2-3p-\frac{3p^2}{8x^2}-p+\frac{3p^2}{8x^2}\\
 &=6x^2-4p\\
 &=6(p-4y^2)-4p
 =2p-24y^2\pmod{p^3}.
\end{align*}
The last equality uses \(x^2=p-4y^2\), proving
\eqref{eq:minus-eight-aux}.
\end{proof}

\begin{proof}[Proof of Theorem~\ref{thm:minus-eight}]
Equations~\eqref{eq:negative-cubic} and
\eqref{eq:negative-catalan-sum} give
\begin{align}
 \sum_{k=0}^m\frac{C_k^3}{(-8)^k}
 &\equiv-1-15F+54H+24J\notag\\
 &\equiv-1-15F+18(p-F)+24J\notag\\
 &=-1+18p-33F+24J\pmod{p^3},
 \label{eq:negative-main-reduction}
\end{align}
where \(54H\equiv18(p-F)\pmod{p^3}\) follows from
\eqref{eq:known-negative-weighted}.
Insert \eqref{eq:known-F} and Theorem~\ref{thm:minus-eight-aux}
into \eqref{eq:negative-main-reduction}:
\begin{align*}
 \sum_{k=0}^m\frac{C_k^3}{(-8)^k}
 &\equiv-1+18p
       -33\left(4x^2-2p-\frac{p^2}{4x^2}\right)
       +24(2p-24y^2)\\
 &=-1-132x^2-576y^2+132p+\frac{33p^2}{4x^2}\\
 &=-1-48y^2+\frac{33p^2}{4x^2}\pmod{p^3}.
\end{align*}
For the last equality,
\(-132x^2+132p=132(p-x^2)=528y^2\), and
\(528y^2-576y^2=-48y^2\).

It remains to put the correction term in the stated form. The exact
identity
\begin{align*}
 \frac{33p^2}{4x^2}+\frac{33p^2}{16y^2}
 &=\frac{33p^2(x^2+4y^2)}{16x^2y^2}\\
 &=\frac{33p^3}{16x^2y^2}
\end{align*}
shows that the left side vanishes modulo \(p^3\), since
\(p\nmid16x^2y^2\). Thus
\[
 \frac{33p^2}{4x^2}
 \equiv-\frac{33p^2}{16y^2}\pmod{p^3}.
\]
Substitution proves \eqref{eq:minus-eight-main}.
\end{proof}


\end{document}